\documentclass[12pt]{article}

\usepackage{latexsym,amssymb,amsmath}

\usepackage[dvips]{graphicx}

\newcommand{\C}{\mathbb C}
\newcommand{\Z}{\mathbb Z}
\newcommand{\N}{\mathbb N}

\newcommand{\Q}{\mathbb Q}

\newcommand{\sma}{\left(\begin{array}}
\newcommand{\fma}{\end{array}\right)}

\newtheorem{lem}{Lemma}[section]

\newtheorem{co}[lem]{Corollary}
\newtheorem{thm}[lem]{Theorem}
\newtheorem{prop}[lem]{Proposition}

\newenvironment{proof}{\textbf{Proof.}}{\newline\hspace*{\fill}{$\Box$}\\}

\begin{document}
\title{Tubular free by cyclic groups and the strongest Tits alternative}
\author{J.\,O.\,Button}

\newcommand{\Address}{{
  \bigskip
  \footnotesize

\textsc{Selwyn College, University of Cambridge,
Cambridge CB3 9DQ, UK}\par\nopagebreak
  \textit{E-mail address}: \texttt{j.o.button@dpmms.cam.ac.uk}
}}

\maketitle
\begin{abstract}
We show, using Wise's equitable sets criterion, that every tubular free by
cyclic group acts freely on a CAT(0) cube complex. We also show that these
groups have a finite index subgroup satisfying the strongest Tits 
alternative, which means that every subgroup either surjects a non abelian
free group or is torsion free abelian. In particular the Gersten group is
the first known group virtually having this property but which is not
virtually special nor virtually residually free.
\end{abstract}

\section{Introduction}
When dealing with finitely presented groups, it may well be the case that
knowledge of how they act geometrically allows us to prove purely group
theoretic facts. An extreme example of this occurs with a word hyperbolic 
group, because if it acts properly and cocompactly on a CAT(0) cube
complex (which will therefore be finite dimensional) we can conclude by the
Agol - Wise results that it has a finite index subgroup which embeds in a
right angled Artin group (a RAAG). In this case we say our group is virtually
special and then any property that holds for all subgroups of RAAGs will
hold virtually for our hyperbolic group. Two useful examples of this are
linearity and (if non abelian) largeness. In fact very strong versions of
both hold in this case; first we obtain linearity not just over $\C$ but
over $\Z$. As for largeness, it was shown in \cite{yagash} that RAAGs have
the property that any subgroup (even if not finitely generated) either
surjects the free group $F_2$ or is torsion free abelian, which they named
the strongest Tits alternative. Thus our hyperbolic group possessing such
an action will have a finite index subgroup satisfying the strongest Tits 
alternative. However hyperbolicity is certainly needed here to go from this
geometric action to these group theoretic properties. If we remove this 
condition then the Burger - Mozes groups act properly and cocompactly
on a 2 dimensional CAT(0) cube complex, but they can be simple groups
which will fail nearly all of the usual group theoretic properties of
interest.

In this paper we will consider two families of groups which all contain
$\Z^2$ and therefore cannot be hyperbolic but which are generally
considered to be well behaved. The first is free by cyclic groups, or
more accurately (finite rank free) by $\Z$ groups. It was proved in
\cite{wshg2} that if a free by cyclic group is word hyperbolic then it
does act properly and cocompactly on a CAT(0) cube complex and so has
the above strong group theoretic properties. Now it is known that if a
free by cyclic group is not word hyperbolic then it contains $\Z^2$, so
we can ask: does such a group always have a ``nice'' geometric action on
a CAT(0) cube complex? Of course it depends on what is meant by nice
but in \cite{ger} Gersten displayed a free by cyclic group that cannot
act properly and cocompactly on any CAT(0) space. Moreover this group
is not virtually special but we can still ask whether the strong properties
of virtually special groups also hold for the Gersten group, even though
we will not be able to establish them by finding such geometric actions.
Here we have nothing to say about linearity but we do show that the Gersten
group virtually satisfies the strongest Tits alternative (though virtually
is needed here).

The other class of groups we consider in this paper are what have been
called the tubular groups: namely the fundamental group of a finite graph
of groups with all vertex groups isomorphic to $\Z^2$ and all edge groups
isomorphic to $\Z$. These have been considered from both a geometric
and group theoretic point of view. In \cite{bribr} they were shown to
have interesting Dehn functions and their quasi-isometric classes were
considered in \cite{csh}. Now Wise showed in \cite{wsja} that they  
can fail to be Hopfian groups, so certainly they can be badly behaved
group theoretically as they need not be linear or even residually finite.
As for the existence of nice geometric actions, the Gersten group is
in fact also a tubular group, so proper and cocompact actions on a CAT(0)
cube complex will not always exist. However in \cite{wscb} Wise looked
at the question of when a tubular group has a free action on a CAT(0)
cube complex. As all our groups are torsion free and because
proper here means topologically proper (a compact set can only have 
finitely many elements which translate it to an image that intersects this
set), acting properly and acting freely mean the same thing in this context.
In \cite{wscb} Corollary 5.10 the tubular groups acting freely and
cocompactly on a CAT(0) cube complex are classified, though these are
quite restrictive. However, on removing the compactness hypothesis,
the main result in Theorem 1.1 of that paper gives a condition (in terms of 
what are called equitable subsets of the vertex groups) 
that determines exactly when a
tubular group acts freely on a CAT(0) cube complex (though this complex
might not be finite dimensional, nor locally finite). From this he was
able to show that a wider range of tubular groups have a free action
on a CAT(0) cube complex, in particular his non Hopfian example and
the Gersten group, but he also used this condition to
give examples of tubular groups with no such actions.
Afterwards in \cite{wd1} a criterion for when a
tubular group acts freely on a finite dimensional CAT(0) space was
developed, and in the forthcoming paper \cite{wd2} it is shown that this
is equivalent to the group being virtually special.

Here our purpose is to examine the intersection of these two classes of
groups, namely the tubular groups $G(\Gamma)$
which are also free by cyclic. We
first identify in Theorem \ref{fbyc} exactly which tubular groups are
free by cyclic, which are those groups having a homomorphism to $\Z$ that
is non zero on all edge groups. This always holds if the underlying graph
$\Gamma$ is a tree. In \cite{wscb}, Conjecture 1.8 states
that every free by cyclic group acts freely on a CAT(0) cube complex.
Using our homomorphism we can show in Theorem \ref{free} that any
tubular free by cyclic group satisfies Wise's equitable subsets condition
and therefore does have a free action on a CAT(0) cube complex.

However, as mentioned above, this result will not allow us to
establish group theoretic properties for tubular free by cyclic groups such
as the Gersten group. Thus in Section 3 we look at whether tubular free
by cyclic groups satisfy the strongest Tits alternative. Although they
need not in general, they are not far from doing so in that they all
have a finite index subgroup that does. In fact we consider tubular groups
where all edge groups embed as maximal cyclic subgroups of the vertex
groups, as it is straightforward to establish in Proposition \ref{fincv}
that any free by cyclic tubular group has a finite cover of this form.
We then show in Theorem \ref{main} that all tubular groups with
maximal edge inclusions do virtually satisfy the strongest Tits alternative,
though virtually is needed here. The method of proof is to look purely
at the graph of groups and the associated action on the Bass - Serre tree,
resulting in a parity argument that is established at the end of the
proof, once the appropriate finite index subgroup is revealed. 

In Section 4 we take a quick look at some other group theoretic properties
which are implied if the group is virtually special, but which also
hold for all or many tubular groups. Largeness is quickly established
for all tubular groups, regardless of the edge embeddings, in 
Proposition \ref{lrg} and virtual biorderability for all free by cyclic
tubular groups in Theorem \ref{vbi}, using the Perron - Rolfsen criterion
in \cite{perrlf03}. The other property considered here is that of being
residually free, whereupon it is pointed out at the end of this Section that
very few tubular groups are residually free or even virtually residually
free. The reason for mentioning this negative result is that a group which
is virtually residually free will of course virtually satisfy the strongest
Tits alternative. In particular we have shown that the Gersten group is
the first group known virtually to satisfy the strongest Tits alternative
which is not either virtually special or virtually residually free.

We finish in Section 5 with a range of examples of tubular groups that have
already appeared in the literature, with comments on which group theoretic
properties are known to hold for them. The first three are all free by
cyclic groups and it will be seen that a whole host of strong group
theoretic properties hold for all tubular free by cyclic groups (though
not subgroup separability, whereas linearity is still in question).
Meanwhile the other three examples seem to behave in a less predictable
way. A straightforward method to create tubular groups with 
bad properties is to ensure the existence of non Euclidean Baumslag - Solitar
subgroups (namely groups with presentation $\langle a,t|ta^mt^{-1}=a^n\rangle$
where $|m|\neq |n|$). However these last three examples do not contain such
subgroups, and in particular it is not the case that every tubular group
either contains a non Euclidean Baumslag - Solitar group (``badly behaved'')
or has a finite index subgroup which is free by cyclic (``well behaved''),
but there are also ``strangely behaved'' tubular groups too.

\section{Tubular free by cyclic groups acting freely}

Given a finite graph of groups $G(\Gamma)$
with all vertex groups isomorphic
to $\Z^2$ and all edge groups isomorphic to $\Z$, 
we can produce a presentation for the resulting 
fundamental group $G$, which will be referred to as a tubular group,
in the usual way. 
We first pick a maximal tree in $\Gamma$ and contract each edge by forming
an amalgamated free product. As $\Z^2$ has an obvious 2-generator
1-relator presentation and we need to add 1 relator each time
when performing the amalgamation, 
this process creates a presentation with $2v$ generators and $2v-1$
relations if there are $v$ vertices. Having done this, we then
introduce a stable letter for each of the $b$ edges left ($b$ being
the first Betti number of the graph $\Gamma$) and form HNN extensions
identifying the remaining cyclic subgroups, thus resulting in a
presentation for $G$ which has $2v+b$ generators and
$2v+b-1$ relators. In particular $G$ has a presentation of
deficiency 1, that is where the number of generators is 1 more than
the number of relators.

Moreover any finite presentation for $G$ has deficiency at
most 1, which can be seen because the standard presentation 2-complex
forms an aspherical graph of spaces as in \cite{scwl} Section 3.
Thus $G$ is also of cohomological and geometric dimension 2 and its
Euler characteristic will equal 0, where 1 minus the Euler characteristic
is an upper bound for the deficiency of $G$. Also $G$ is well known to
be coherent, namely every finitely generated subgroup of $G$ is finitely
presented.

Another well
behaved class of groups sharing these nice properties are the free by
$\Z$ groups, which here refer to finitely presented groups of the
form $F\rtimes_\alpha\Z$ where $F$ is free of finite rank. Therefore
it is of interest to determine when a tubular group is actually
free by $\Z$, with the next theorem giving us a complete answer.

\begin{thm} \label{fbyc}
A tubular group 
$G$ is isomorphic to a free by $\Z$ group if and only
if there exists a homomorphism from $G$ to $\Z$ which is
non zero on every edge group.
\end{thm}
\begin{proof}
We use the following two results:
\begin{prop} \label{rat}
(i) (\cite{rat} Corollary 1)
Let $G$ be a group with an amalgamated product decomposition $G=A\ast_CB$, 
and  let $\phi: G \to \Z$ be a homomorphism such that ${\rm ker}
(\phi |_C)$ is finitely generated 
and not equal to $C$. 
Then ${\rm ker}(\phi)$ is finitely generated if and only if 
${\rm ker}(\phi |_A)$ and ${\rm ker}(\phi |_B)$ 
are finitely generated.\\ 
(ii) (\cite{brhw} Proposition 2.2)
If $G$ splits over the edge group $C$ 
and $N$ is finitely generated and normal in $G$ then
either $N$ is in $C$ or else $CN$ has finite index in $G$.
\end{prop}
First suppose that we have such a homomorphism $\chi$. We begin with
the case where our graph $\Gamma$ is a tree and work by induction on 
the number of vertices by taking an amalgamated free product at each stage.
Thus let us say that we have the subgroup $A$ of $G$ formed by
amalgamation of the vertices in a subtree of $\Gamma$. Our inductive 
hypothesis is that ${\rm ker}(\chi |_A)$ is finitely generated and free,
whereupon on making the next amalgamation we are forming
$H=A\ast_CB$ with $B\cong\Z^2$ and $C\cong\Z$. Now $\chi$ is not zero
on the edge group $C$, so that  
${\rm ker}(\chi |_C)$ is certainly finitely generated 
and not equal to $C$ as it is trivial. Moreover 
${\rm ker}(\chi |_A)$ and ${\rm ker}(\chi |_B)$ 
are finitely generated, so Proposition \ref{rat} (i) tells us that
${\rm ker}(\chi |_H)$ is too.
To see that it is free, we have by \cite{tre} Section I.5.5
Theorem 14 that a subgroup
$S$ of $H=A\ast_CB$ that misses all conjugates of $C$ is a free product
of a free group with factors of the form $S\cap hAh^{-1}$,
$S\cap hBh^{-1}$ for various elements $h\in H$. But on setting 
$S={\rm ker}(\chi |_H)$, we have that $S\cap hAh^{-1}$ is equal
to $h{\rm ker}(\chi |_A)h^{-1}$, which is free by hypothesis. As for
$S\cap hBh^{-1}\cong{\rm ker}(\chi |_B)\leq B\cong\Z^2$, this is trivial
or isomorphic to $\Z$ because $C$ is not in the kernel. 

If now $\Gamma$ is not a tree then we can still perform the above
process on a spanning tree. On contracting this tree we are left
with a single vertex where we have a subgroup $H$
of $G$ which is free by $\Z$ 
because ker $(\chi |_H)$ is
a free group of finite rank $r$ say,
along with some self loops at this vertex corresponding to HNN
extensions. Let $t$ be an element of $H$ generating the image
 of $\chi$ so that we can express $H$ as
\[K\rtimes_\alpha\langle t\rangle=
\langle k_1,\ldots,k_r,t|tk_1t^{-1}=\alpha(k_1),\ldots ,tk_rt^{-1}=
\alpha(k_r)\rangle\]
for the appropriate automorphism $\alpha$ of 
ker $(\chi |_H)$ which we now call $K$. Now when we take a loop
and form the HNN extension $E$
of $H$ with stable letter $s$ conjugating
$x\in H$ to $y\in H$, we can write $x=ut^m$ and $y=vt^n$ for $u,v\in K$.
But $sxs^{-1}=y$ implies that $m=n$ on application of $\chi$, which is
defined on $E$ too. Moreover
$n\neq 0$ as $x$ generates an edge group. Thus we have
\[E=\langle s,k_1,\ldots,k_r,t|tk_1t^{-1}=\alpha(k_1),\ldots ,tk_rt^{-1}=
\alpha(k_r),sut^ns^{-1}=vt^n\rangle\]
which, on rearranging the last relation and introducing new generators
$s_1=s$, as well as $s_i=ts_{i-1}t^{-1}$ for $2\leq i\leq n$ when $n\geq 2$,
admits the following presentation:
\begin{eqnarray} \label{eq1} E=\langle 
k_1,\ldots,k_r,s_1,\ldots ,s_n,t
&|&tk_1t^{-1}=\alpha(k_1),\ldots ,tk_rt^{-1}=\alpha(k_r),\\
&&ts_1t^{-1}=s_2,\ldots ,
ts_{n}t^{-1}=v^{-1}s_1u\rangle   \nonumber
\end{eqnarray}
whereupon we see that $E$ is also free by $\Z$. We can now remove the
self loops one at a time to conclude that $G$ is also
free by $\Z$.

Finally suppose that we have any homomorphism $\chi$ from a given
tubular group $G(\Gamma)$ to $\Z$ with kernel $K$, along with an
edge group $E$ with $\chi(E)=0$. We can form $G$ by leaving
this edge until last, whereupon either $G=A*_EB$ or
$G=H*_E$. Now Proposition \ref{rat} (ii) implies that if
$K$ is finitely generated then, as $E\subseteq K$ so $EK$ has infinite
index in $G$, we have $E=K\cong\Z$ which implies
that $G=\Z^2$ and $\Gamma$ is a single vertex.
We now let $\chi$ vary over all homomorphisms from $G$ to
$\Z$.
\end{proof}
Note: more recently it was shown in \cite{cl} how to calculate the 
BNS invariant of a graph of groups, thus telling  us which homomorphisms
have finitely generated kernel. 
\begin{co} \label{tree}
If $\Gamma$ is a tree then $G(\Gamma)$ is free by $\Z$.
\end{co}
\begin{proof}
We pick a vertex $v_0$ with which to start defining $\chi$, which
for now will be a homomorphism from $G$ to $\Q$,
and then
inductively extend to the vertices at distance $n$ from $v_0$. Say
this has been done in such a way that $\chi$ is non zero on every
edge group and take a vertex $v$ at distance $n+1$. We can let
this vertex group equal $\langle x,y\rangle$ where the edge group
into $v$ from distance $n$ is generated by a power of $x$ and thus 
we have $\chi(x)\neq 0$. We then have other generators for the edge
groups out of $v$ which will all be of the form $x^iy^j$. As this
is a finite list, we can choose a value of $\chi(y)\in\Q$ such that
no edge group in $v$ is sent to zero by $\chi$. We then extend $\chi$
to all of the vertex groups, hence to $G$, 
and this will be well defined because $\Gamma$ is
a tree.

We end up with a homomorphism $\chi$ from a finitely generated group to $\Q$,
thus the image is cyclic and we can multiply $\chi$ by a suitable
integer so that this image now lies in $\Z$.
\end{proof} 
Next we examine the necessary and sufficient condition stated in 
\cite{wscb} Theorem 1.1
for a tubular group to act freely on a CAT(0) cube complex (which
in general here need not be finite dimensional nor locally finite).
This condition is that there exists an equitable set, which is a choice 
of a finite family of elements from each vertex group such that\\
(i)
each family generates a finite index subgroup of the respective vertex
group and\\
(ii) the compatibility condition is satisfied on each edge $e$, namely
on taking a generator of this edge group which embeds as 
${\bf x}$ in
one vertex group and ${\bf y}$ in the other, we have
\[\sum_{i=1}^m \#[{\bf x},{\bf s_i}]=\sum_{j=1}^n \#[{\bf y},{\bf t_j}].\]
Here $\{{\bf s_1},\ldots ,{\bf s_m}\}$ is the chosen family for the 
first mentioned
vertex group, $\{{\bf t_1},\ldots ,{\bf t_n}\}$ is for the second 
and $\#[{\bf x},{\bf s}]$ is
the intersection number. If we have 
${\bf x}=(x_1,x_2)$ and ${\bf s}=(s_1,s_2)$
for some choice of basis of the first vertex group then 
$\#[{\bf x},{\bf s}]$ is equal
to the modulus of the determinant $x_1s_2-x_2s_1$ of the matrix
$\sma{cc}x_1&s_1\\x_2&s_2\fma$, so that we will write
$|({\bf x}\,\,{\bf s})|$ for $\#[{\bf x},{\bf s}]$.  

It was shown in \cite{wscb} Example 1.4
that the Gersten free by cyclic tubular
group (with graph a single vertex and a bouquet of two circles) satisfies
the above condition, thus although it does not act properly and cocompactly
on any CAT(0) metric space by \cite{ger}, or even properly and semisimply
by \cite{bh}, it does act 
freely on a CAT(0) cube complex.
Conjecture 1.8 of this paper says that every free by $\Z$ group possesses
such an action. Here we are able to confirm this conjecture for the
tubular free by $\Z$ groups. 
\begin{thm} \label{free}
If $G(\Gamma)$ is a tubular free by $\Z$ group then there
exists an equitable set for $G(\Gamma)$ and so $G$ acts freely on a
CAT(0) cube complex.
\end{thm}
\begin{proof} 
We use Theorem \ref{fbyc}, so that we have $\chi:G\rightarrow\Z$
which is non zero on all the edge groups. 

Our first case is when $\Gamma$ is a bouquet of circles, so that we have
only one vertex $v$ but many self loops.
We pick a basis for our vertex group $v$ and therefore represent the elements
as ordered pairs of integers. Therefore suppose that, having directed the
$k$ edges arbitrarily, we have the following generators when the edge groups
are injected via the negative and the positive ends respectively:
\[{\bf a_1}=(a_1,b_1),\ldots ,{\bf a_k}=(a_k,b_k)\mbox{ and }
{\bf c_1}=(c_1,d_1),\ldots ,{\bf c_k}=(c_k,d_k)\]
which are all elements of $\Z^2$. We now require a choice of
$N$ elements 
${\bf x_1},\ldots ,{\bf x_N}\in\Z^2$
such that for for each $j$ with $1\leq j\leq k$ the following
equation holds:
\begin{eqnarray}
\sum_{i=1}^N |({\bf x_i}\,\,{\bf a_j})|&=&\sum_{i=1}^N 
|({\bf x_i}\,\,{\bf c_j})|.
\end{eqnarray}
We will be able to do this with $N=2$ by proceeding as follows. 
First note that we can work over $\Q^2$ instead of $\Z^2$ throughout,
as on finding ${\bf x_1},\ldots ,{\bf x_N}\in\Q^2$ satisfying (2), we
can clear out denominators by multiplying both sides of (2) by an
appropriate integer. 

Suppose
that the generators of the vertex group $(1,0)$ and $(0,1)$ are mapped
to $m$ and $n$ respectively under $\chi$, with the basis chosen such that
$m\neq 0$. Set $l_j$ to be $\chi({\bf a_j})=ma_j+nb_j$ so that also
$l_j=\chi({\bf c_j})=mc_j+nd_j$, with no $l_j$ being equal to 0 by the
properties  of $\chi$. On setting $l=l_1$, the idea is to consider as
candidates for the ${\bf x_i}$ points of the form
$(\frac{l-ny_i}{m},y_i)$ for $y_i\in \Q$ to be determined. Note that
as all of these points lie on the line $mx+ny=l\neq 0$, these will span
$\Q^2$ provided only that we take more than one distinct point of this
form. We find for each $j$ between 1 and $k$ that
\[\sum_{i=1}^N|({\bf x_i}\,\,{\bf a_j})|
=\sum_{i=1}^N
\left|\sma{cc}\frac{l-ny_i}{m}&\frac{l_j-nb_j}{m}\\
         y_i&b_j\fma\right|
=\sum_{i=1}^N\frac{|lb_j-l_jy_i|}{|m|}
=\frac{|l_j|}{|m|}\sum_{i=1}^N|y_i-l\frac{b_j}{l_j}|\]
and so we also have
\[\sum_{i=1}^N|({\bf x_i}\,\,{\bf c_j})|
=\frac{|l_j|}{|m|}\sum_{i=1}^N|y_i-l\frac{d_j}{l_j}|.
\] 

Now if given two $k$-tuples of rationals $(q_1,\ldots ,q_k)$ and
$(r_1,\ldots ,r_k)$, we can find $N$ and $y_1,\ldots ,y_N\in\Q$
such that
\[
\sum_{i=1}^N|y_i-q_j|=
\sum_{i=1}^N|y_i-r_j|
\]
holds for each $j$ between 1 and $k$ by setting $N=2$, then
letting $y_1$ be the minimum $m$ of our $2k$ rational entries
and $y_2$ be the maximum $M$ whereupon both sums become $M-m$
regardless of $j$. (In order to ensure that $y_1\neq y_2$ for
spanning, we can increase $M$ and decrease $m$ if needed.)

For a general graph $\Gamma$, we will reduce to the above case
by contracting edges one by one in a maximal tree. 
However this will not be the usual process used when calculating
the fundamental group of a graph of groups. Rather, given an
oriented edge $e$ running from the vertex $v\in V(\Gamma)$ to
the vertex $w\neq v$, we will replace the graph of groups
$G(\Gamma)$ by another graph of groups $G_1(\Gamma_1)$ where
although $\Gamma_1$ is the result of contracting the edge $e$
in $\Gamma$ onto the vertex $w$, the fundamental groups $G$ and
$G_1$ will in general be different. The procedure is not to
amalgamate the vertex groups $G_v$ and $G_w$ but rather to replace
$G_v$ with $G_w$ using a suitable isomorphism between them that
respects the homomorphism $\chi$. We then reinterpret the inclusions
into $G_v$ of the edge groups for the other edges ending at $v$ as
inclusions into $G_w$ using this isomorphism.

We start by choosing
a basis for each vertex group so that all the elements in a particular
vertex group can be thought of as lying in $\Z^2$. However when finding
our equitable vectors, we again work
over $\Q^2$ and convert these into
elements of $\Z^2$ by clearing out denominators at the end.

Given the first edge $e_1$ to be contracted, with vertex groups
$V_1^+$ and $V_1^-$ at the vertices $v_1^+,v_1^-$ of $e_1$ which are now
both isomorphic to $\Q^2$, let $g_1^+$ be the
injection into $V_1^+$ of the generator of the edge group at $e_1$, and
let $k_1^+$ be a non zero element of $V_1^+$ such that $\chi(k_1^+)=0$.
Note that $\chi(g_1^+)\neq 0$ by hypothesis, so $g_1^+$ and $k_1^+$
span $\Q^2$.
We also do the same at the other end of $e_1$
to obtain two elements $g_1^-$ and $k_1^-$ of $V_1^-$ and we then
define the isomorphism $\theta_1:V_1^-\rightarrow V_1^+$ by sending
$g_1^-$ and $k_1^-$ to $g_1^+$ and $k_1^+$ respectively. However, as
the elements $k_1^{\pm 1}$ are only defined up to multiplication by
a non zero scalar, we choose them such that the
matrix $P_1$, which represents $\theta_1$ with respect to the original bases 
chosen for the vertex groups,  has determinant 1 in modulus. (This will
mean that $P_1$ has entries in $\Q$, though not necessarily in $\Z$.)
Moreover as $\chi(g_1^+)=\chi(g_1^-)$ and $\chi(k_1^+)=\chi(k_1^-)$
we have that $\chi(\theta_1(\lambda g_1^-+\mu k_1^-))=\chi(\lambda g_1^++
\mu k_1^+)$ on temporarily writing these groups additively.

We now contract the edge $e_1$ from the vertex $v_1^-$ to $v_1^+$. This means
that all other edges (if any) meeting $v_1^-$ get moved
to meet $v_1^+$ and we replace each
generator of the image of these edge groups in $v_1^-$ by applying $\theta_1$
to this element. This results in the graph $\Gamma_1$ which is the
result of contracting the edge $e_1$ in $\Gamma$, along with cyclic
edge groups and embeddings of these into the neighbouring vertex groups
except that we can have fractional powers of elements. Note that as
$\chi$ respects this substitution, we still have a homomorphism from the new
fundamental group $G_1$ to $\Z$ which is non zero on all edge groups and
we call this $\chi$ as well.

We now continue this process, producing further isomorphisms
$\theta_2,\theta_3,\ldots ,$ $\theta_{l-1}$ until $\Gamma$ has been
contracted to the graph $\Gamma_l$ with just one vertex $v_l$, though
there could be many self loops at this vertex. By the previous argument,
we can find an equitable set $\{{\bf x_1},{\bf x_2}\}$ consisting of
a pair of elements that are both in $\Q^2$ for the labelled graph
$\Gamma_l$. This means that for each of the $k$ edges and pair of
elements ${\bf a_j},{\bf c_j}\in\Q^2$ at either end of the $j$th edge,
we have
\begin{eqnarray} 
|({\bf x_1}\,\,{\bf a_j})|+|({\bf x_2}\,\,{\bf a_j})|
&=& |({\bf x_1}\,\,{\bf c_j})|+|({\bf x_2}\,\,{\bf c_j})|.
\end{eqnarray}
(If $\Gamma$ were a tree then we would have no edges left so in this case
we would not contract the final edge. It is then straightforward to
find an equitable set consisting of a pair of elements for each of the
two remaining vertices.)
 
Now we
reverse our contracting process and our isomorphisms. This means
that if immediately prior
to the application of the final isomorphism $\theta_{l-1}$,
we have that there is some $j$ where 
${\bf a_j}$ is
based at $v_l$ but ${\bf c_j}$ is based at the other vertex $v_{l-1}$,
we replace ${\bf c_j}$ by the element $\theta^{-1}_{l-1}({\bf c_j})$
so that this element now lies back in the vertex group $V_{l-1}$.
However we also
pick the elements $\theta_{l-1}^{-1}({\bf x_1})$ and 
$\theta_{l-1}^{-1}({\bf x_2})$ for the part of our equitable set
based at $v_{l-1}$. To see this does make an equitable set, note
that we already had equation (3) holding
and we are now replacing (for $i=1,2$)
${\bf x_i}$ by $P_{l-1}^{-1}({\bf x_i})$ and
${\bf c_j}$ by $P_{l-1}^{-1}({\bf c_j})$, because
$P_{l-1}$ is the matrix representing $\theta_{l-1}$.
But for any 2 by 2 matrix $M$ and ${\bf x},{\bf y}\in \Q^2$ 
we have 
\[|(M{\bf x}\,\,M{\bf y})|=|M({\bf x},{\bf y})|=|M||({\bf x},{\bf y})|\]
and as here $M=P_{l-1}^{-1}$ has determinant $\pm 1$, we see that the
right hand side of (3) is unchanged when we insert the new elements.
Similarly both sides of (3) are unchanged 
if for another value of $j$ both ${\bf a_j}$ and ${\bf c_j}$
get moved back to the vertex $v_{j-1}$, and if they both stay at $v_j$
then nothing at all in (3) is changed.

We can now reverse this process, introducing equitable sets at each
new vertex until we return to $\Gamma$, whereupon we have our original
tubular group $G(\Gamma)$ but now with a pair of equitable elements at every
vertex. 
\end{proof} 

Combining Theorems \ref{fbyc} and \ref{free}, we immediately obtain:
\begin{co} Any tubular group which is also free by $\Z$ acts
freely on a CAT(0) cube complex.
\end{co}

\section{The strongest Tits alternative}
A group that has a homomorphism onto the free group $F_2$ must necessarily
have a subgroup isomorphic to $F_2$ by the universal property of free
groups.  In \cite{yagash} a group $G$ is said to satisfy the strongest
Tits alternative if every subgroup $H$ of $G$ (not necessarily finitely
generated) either has a homomorphism onto $F_2$ or is a torsion
free abelian group. The authors showed in this paper that right angled
Artin groups (and hence their subgroups) satisfy the strongest Tits
alternative. Thus any group that is known to be virtually special,
that is having a finite index subgroup which embeds in a right angled
Artin group, will virtually satisfy the strongest Tits alternative.
Here we are interested in establishing this property for
various tubular groups. However a tubular group need not have
a finite index subgroup which embeds in a right angled Artin group,
for instance we will see this in Example 5.2
for the Gersten group, thus we will need an alternative
argument.

Our first problem is that  
given a graph of groups $G(\Gamma)$ with $\Z^2$ vertex groups and $\Z$ edge
groups, it might well be that the inclusion of an edge group
$\langle x\rangle$ in one of its vertex groups is not maximal, meaning that 
there is $g\in\Z^2\setminus\{id\}$ and $n>1$ such that $x=g^n$ (which
is of course not the same as saying that $\langle x\rangle$ is a maximal
subgroup of $\Z^2$). Certainly these can give rise to subgroups of 
the fundamental group $G$ which are not abelian but which do not
surject to $F_2$, even if $\Gamma$ is a tree. For instance
\[\langle a,b,x,y,c,d|[a,b],[x,y],[c,d],b^2=x,c^3=y\rangle\]  
has the subgroup $\langle b,x,y,c|[x,y],b^2=x,c^3=y\rangle$ which is
non abelian (as it surjects to the modular group $C_2*C_3$) but any
surjection to a free group would have to send $\langle x,y\rangle$ to
a cyclic subgroup so the image of this surjection would have first
Betti number 1. However we will see that actually this group has a finite
index subgroup obeying the strongest Tits alternative.

A more extreme situation occurs when the fundamental group $G$ contains a
non Euclidean Baumslag-Solitar subgroup, namely a subgroup isomorphic
to $\langle a,t|ta^mt^{-1}=a^n\rangle$ for $m,n\neq 0$ and where 
$|m|\neq |n|$. These groups do not surject $F_2$ and moreover, by
replacing $a$ with $b=a^d$ where $d$ is the highest common factor of
$m$ and $n$, we have a further subgroup isomorphic to
$B=\langle b,t|tb^{m/d}t^{-1}=b^{n/d}\rangle$ where $m/d$ and $n/d$ are
coprime. If $G$ contains such a subgroup $B$ then no finite index subgroup
$H$ of $G$ can satisfy the strongest Tits alternative. This is because
$B$ is not large, so the finite index subgroup $H\cap B$ of $B$ cannot
surject $F_2$ but it is never abelian either.

However both these examples use non maximal inclusions of edge groups.
In \cite{me} we gave an exact criterion for a finite graph of groups
with $\Z$ edge groups and a wide variety of possible vertex groups,
certainly including $\Z^n$, to contain a non Euclidean Baumslag-Solitar
group and it is immediate from this that it never happens if all edge
inclusions are maximal. Moreover it also cannot occur if the
fundamental group is free by $\Z$ (which covers the case when $\Gamma$
is a tree by Corollary \ref{tree}), but there are examples where all edge
inclusions are maximal and the fundamental group is free by $\Z$ but
does not surject $F_2$. For instance the Burns group
$\langle x,y,t|[x,y],txt^{-1}=y\rangle$ in Example 5.1 where the graph is
a single self loop does not surject $F_2$ as it is 2-generated but not free 
and $F_2$ is Hopfian.
Therefore we can at least hope that the fundamental group virtually
satisfies the strongest Tits alternative in
the two cases where it is free by $\Z$ and also where all edge inclusions are
maximal.  The rest of this section will provide a proof of this, starting
by reducing the first case to the second.
\begin{prop} \label{fincv}
Let $G(\Gamma)$ be a finite graph of groups 
with $\Z^2$ vertex groups and $\Z$ edge groups such that $G$ is free by $\Z$.
Then there exists a finite index subgroup $H$ of
$G$ which can be written as a finite graph of groups 
$H(\Delta)$ with $\Z^2$ vertex groups and $\Z$ edge groups where all
inclusions of edge groups into vertex groups are maximal.
\end{prop}
\begin{proof}
As we have a homomorphism $\chi:G(\Gamma)\twoheadrightarrow\Z$ which is 
non zero on every edge group by Theorem \ref{fbyc}, 
for a generator $x_i$ of each edge group $\langle x_i\rangle$ we can set
$\chi(x_i)=m_i\neq 0$ and let $M$ be the lowest common multiple over all
of the finitely many $m_i$. We take for $H$ the kernel of the map
$\chi$ modulo $M$ which is a surjective homomorphism from $G$ to the
cyclic group $C_M$. Consider the action of $G(\Gamma)$ on its Bass - Serre
tree $T$, where of course we have $\Z^2$ vertex stabilisers and $\Z$
edge stabilisers and such that there are only finitely many orbits of
vertices and edges under $G(\Gamma)$. Now $H$ also acts on this tree
with vertex stabilisers $H\cap \Z^2$ and edge stabilisers $H\cap\Z$, which
are isomorphic to $\Z^2$ or $\Z$ respectively because $H$ has finite
index in $G$. Thus by taking the quotient $H\backslash T$ we see that
$H$ is also a graph of groups $H(\Delta)$
with $\Z^2$ vertex groups and $\Z$ edge groups, and $\Delta$ is a finite
graph because $H$ having finite index in $G$ also means there are
finitely many vertices and edges in $\Delta$.

Finally we consider an edge subgroup $\langle x_i\rangle$ in $G$ with
$\chi(x_i)=m_i$ and we set $d_i=M/m_i$. This tells us that
$\langle x_i\rangle\cap H$ is $\langle x_i^{d_i}\rangle$ and so in the
graph of groups $H(\Delta)$ there will be two inclusions of $x_i^{d_i}$,
each of which is into a vertex subgroup isomorphic to $\Z^2$. But if
$x_i^{d_i}=h^n$ for $h\in\Z^2\subseteq H$ then $M=|n|\cdot |\chi(h)|$, but
$\chi(h)$ must be zero modulo $M$ so $|n|=1$ and $x_i^{d_i}$ is a maximal
element of this vertex subgroup of $H(\Delta)$.
\end{proof} 

In looking for homomorphisms from the fundamental group of a
graph of groups onto $F_2$, we
begin with the following straightforward but useful lemma, giving us
homomorphisms onto $\Z$. 

\begin{lem} \label{hom}
Suppose that $G(T)$ is a (possibly infinite) graph of groups
with $T$ a tree where all vertex groups are free abelian and
all edge groups embed as direct summands of the vertex groups.
Then any non trivial homomorphism $h$ from a vertex group to
$\Z$ can be extended to a homomorphism $h:G\rightarrow\Z$.
\end{lem}
\begin{proof} Let the root vertex of the tree $T$ be the vertex where
$h$ is defined. This then determines an
integer valued homomorphism on all the edge groups out of the root
vertex. Thus for any vertex joined to the root, we can extend $h$ from
the inclusion of the edge subgroup to the whole vertex group, because
it is a direct summand.
We can now continue this process over
all vertices to get the surjective homomorphism 
$h:G(T)\twoheadrightarrow\Z$, and then take a direct limit if 
$T$ is infinite.
\end{proof}

We can now use this for the following theorem.
\begin{thm} \label{onel}
Let $G(\Gamma)$ be a finite graph of groups having $\Z^2$ vertex groups 
and $\Z$ edge groups and such that all inclusions
of edge groups in vertex groups are maximal. 
Let $S$ be any subgroup of $G$, so that by restricting the action of $G$
on its Bass - Serre tree $T$ we have that $S$
is also a graph of groups $S(\Sigma)$ with vertex stabilisers equal
to $S\cap\Z^2$ and edge stabilisers equal to $S\cap\Z$. Then either
$S$ is free abelian (thus isomorphic to $\{id\}, \Z$ or $\Z^2$), or it
surjects the free group $F_2$, or the graph
 $\Sigma$ contains only one loop.
\end{thm}
\begin{proof}
As $S$ is any subgroup of $G$ it might have infinite index in $G$ or even
be infinitely generated, so we cannot assume $\Sigma$ is a finite graph. 
However in any event it is well known that there is a
surjection from $S$ to the fundamental group $\pi_1(\Sigma)$ of the
underlying graph, so that 
if $\Sigma$ contains at least two loops then 
$S$ surjects to $F_2$ in this case.

Thus we can now suppose that $\Sigma$ is a tree
where the vertex subgroups
of $S(\Sigma)$ are all isomorphic to $\{id\},\Z$ or $\Z^2$ and the edge
subgroups $S\cap\Z$ either are trivial or are copies of $\Z$. Moreover if
here the vertex subgroup $S\cap\Z^2$ is itself isomorphic to $\Z^2$ then we 
would have that $S\cap\Z^2$ has finite index in the original $\Z^2$ vertex
subgroup of $G(\Gamma)$, so therefore any edge subgroup $E$
of $G(\Gamma)$ that
has an inclusion into this vertex subgroup $\Z^2$ will have $S\cap E$
also isomorphic to $\Z$. Furthermore if $\{id\}$ is now a vertex subgroup
in $S(\Sigma)$ then any edge with endpoint this vertex will also have
trivial stabiliser. However in the case when the new vertex subgroup 
$S\cap\Z^2$ is isomorphic to $\Z$, we have that the edge subgroups
including in this vertex subgroup can either be trivial or 
themselves isomorphic to $\Z$.

In summary we have a (possibly infinite) graph of groups $S(\Sigma)$
where the underlying graph $\Sigma$ is actually a tree and
where each vertex group is $\{id\},\Z$ or $\Z^2$, with edge groups
$\{id\}$ joining what we will call $\{id\}$ type vertices, edge groups
$\{id\}$ or $\Z$ joining $\Z$ type vertices and edge groups $\Z$
joining $\Z^2$ type vertices. Moreover we know that in $G(\Gamma)$ all
inclusions of edge groups into vertex groups are maximal and so the same is
true here because we are just intersecting each edge and vertex group
with the same subgroup $S$. 

We begin by assuming here that $\Sigma$ is a finite tree.
If every vertex group of $\Sigma$ is $\Z^2$
then every edge group must be $\Z$. In this case we can
take any edge in $\Sigma$ and by choosing the
relevant bases of the two neighbouring vertex groups, we can suppose
that these two vertex groups are $\langle x,y\rangle\cong\Z^2$ and
$\langle a,b\rangle\cong\Z^2$, where $x$ is amalgamated with $a$. Hence
on removing this edge to obtain two disjoint trees $T_1$ and $T_2$,
we have
the graph of groups $G_1(T_1)$ containing the $\langle x,y\rangle$ vertex
and $G_2(T_2)$ containing the $\langle a,b\rangle$ vertex. Thus we can
apply Lemma \ref{hom} to get
a homomorphism $h$ from $G_1(T_1)$ onto $\Z$ which sends
$x$ to zero and $y$ to 1.
We now do the same with $G_2(T_2)$ and as $x$ and $a$ are sent to zero by
these separate homomorphisms to $\Z$, we have that
$S=G_1*_{x=a}G_2$ surjects to $F_2$.

Indeed if there exists anywhere in $\Sigma$
two $\Z^2$ type vertices which are joined by an
edge (so that this edge group must be $\Z$)  then the argument just given
works here too. Another case that is easily dealt with is when there are
two $\Z$ type vertices which are joined by an edge having trivial
stabiliser. On removal of this edge to form trees $T_1$ and $T_2$ with
notation as before, we now have that $S$ is the free product $G_1*G_2$
and so we can apply Lemma \ref{hom} to each of these trees to conclude
that $G_1,G_2$ both surject $\Z$ and thus $S$ surjects $F_2$.

However there are other cases to consider here and we need to do this
in a way which will extend to infinite trees. First suppose that we
have two distinct vertices $v,w$ of type $\Z^2$ and consider the unique
path in $\Sigma$ running between them. If a vertex $u$ in this
path has type $\{id\}$ then we can ignore it, because 
on contracting one of the two edges that lie in the path and pass through $u$
onto its other endpoint, we are forming
a free product where one of the factors is trivial. As the
stabiliser of both these edges must be trivial, the process
of contracting edges in this path can only end when we obtain an edge
with trivial stabiliser which joins vertices that are both of $\Z$ type
and this has been dealt with above. If however there are no $\{id\}$ type
vertices then either we use the above argument on a $\{id\}$ type edge
which will join two $\Z$ type vertices as before, or a very similar argument
works by removing all $\Z$ type vertices in this path.
This is because the edge groups must now be $\Z$ and
the edge inclusions are maximal, thus 
the inclusion of any edge subgroup is also equal to any vertex group $V_1$
which is of $\Z$ type.
Thus the amalgamation formed by contracting
an edge joining this vertex to a neighbouring vertex with group $V_2$
results in the trivial amalgamation $V_2*_{V_1}V_1=V_2$. 

This covers the case when $\Sigma$ has two $\Z^2$ type vertices but in fact
essentially the same argument applies 
whenever we have a path between vertices $v$ and $w$ with
vertex groups $V,W$ respectively such that both the inclusion into $V$ of
the first edge group and the inclusion into $W$ of the last edge group are
proper. Thus now suppose this never occurs in $S(\Sigma)$. If there is
exactly one $\Z^2$ type vertex then the whole tree can be contracted back
to this vertex without changing the fundamental group, so that $S=\Z^2$.
This also works if there is a $\Z$ type vertex joined to an edge with
trivial stabiliser, leaving only the case where all vertex and edge groups
are isomorphic to $\Z$, in which case $S$ is $\Z$ too as all edge inclusions
are maximal, and the trivial case where all stabilisers equal $\{id\}$.

As for the case where $\Sigma$ is an infinite tree so that $S$ is defined as
a direct limit of the fundamental group of finite subtrees $\Sigma_n$,
we see that once we have taken a finite subtree $\Sigma_N$ that includes a
path with proper edge inclusions at each end, our use of Lemma \ref{hom}
twice to obtain a surjection from the fundamental group of $\Sigma_n$ to
$F_2$ will be consistently defined for $n\geq N$, because each homomorphism
to $\Z$ is defined from the starting vertex outwards, so that this
surjection extends to $S$ as well. If however this never occurs in $\Sigma$
then the fundamental groups stabilise as $\Z^2$, $\Z$ or $\{id\}$.
\end{proof}

Note: the first paragraph of the proof, along with Lemma \ref{hom}, allows
us to conclude that if $G(\Gamma)$ is a graph of groups with free abelian
vertex groups and all edge subgroups embed as direct summands of the vertex
groups then every non trivial subgroup of $G$ has a surjective homomorphism
to $\Z$, thus $G$ is locally indicable, left orderable, has the unique
product property, has no zero divisors in its group algebra and so on.
Moreover the same techniques show that if $G(\Gamma)$ is a finite graph
of groups with finitely generated free abelian vertex groups then, regardless
of the edge embeddings, $G$ is locally indicable and so has the other 
properties. This is because any edge group will embed as a finite index
subgroup of a direct factor, so in the case where the subgroup $S$ is 
such that $S\backslash T$ is a tree
(where $T$ is the Bass-Serre tree for $G$), application of Lemma
\ref{hom} will obtain a non trivial homomorphism $h$ from $S$ to $\mathbb Q$.
But if $S$ is finitely generated then there is a compact subgraph of 
$S\backslash T$ such that the
inclusion of the corresponding graph of groups induces an isomorphism of
fundamental groups. Consequently we can assume that this tree is finite
and so $h$ has discrete image.
    
We now have to go back to finding surjective homomorphisms to $F_2$,
where $\Sigma$ having a single loop is the case that remains, though we have
seen in the example of the Burns group that the conclusion is false if
this loop is a self loop. 
Thus we next look at the case where all vertex groups are $\Z^2$,
all inclusions of edge groups are maximal, and
our graph has a single loop which is not a self loop.
\begin{prop} \label{sloop}
Suppose that $G(\Gamma)$ is a (possibly infinite) graph of groups 
with $\Z^2$ vertex groups and $\Z$ edge groups and such that all inclusions
of edge groups in vertex groups are maximal. If the graph has a single loop
and this is not a self loop then the fundamental group $G$ of the graph of
groups $G(\Gamma)$ surjects to $F_2$.
\end{prop}
\begin{proof}
We can form a graph of groups using just the loop $\Lambda$ in $\Gamma$,
which we will call $L(\Lambda)$, and we will initially define our
homomorphism on the group $L$.

First suppose that there is a vertex $v$ lying in the loop $\Lambda$
with vertex group
$\langle x,y\rangle$ say, having the following property: on taking the
two edges joining $v$ that also lie in $\Lambda$, the inclusions of
these edge groups into $\langle x,y\rangle$ are the same subgroup,
which we will take to be $\langle x\rangle$.
Then we can define a homomorphism from $L(\Lambda)$ to $F_2=\langle u,v\rangle$
as follows: we send
$y$ to $u$ and every other vertex subgroup in $\Lambda$ to the identity, 
along with the element $x$. 
But $L(\Lambda)$ will have another generator,
namely the stable letter $t$ coming from the loop and we send this to $v$
which will provide a well defined surjection from $L$ to $F_2$.

If there is no vertex in $\Lambda$ with this property then we will 
take any edge in the loop and this will span distinct vertex groups
$\langle x,y\rangle$ and $\langle a,b\rangle$, with the edge group
inclusions assumed to be
$\langle x\rangle $ and $\langle a\rangle$ respectively. When forming
$L(\Lambda)$ we will remove this edge to form a maximal tree, so we can
assume that in $L$ we have $txt^{-1}=a$, where $t$ is the stable letter
corresponding to the loop $\Lambda$. We further assume that the 
edge group of the other edge lying in $\Lambda$ and joined to the vertex
with vertex group $\langle x,y\rangle$ is generated by
$x^iy^j$ for coprime $i$ and $j$,
and similarly we have the edge group generated by $a^kb^l$ which is
joined to the $\langle a,b\rangle$ vertex. We suppose that
$F_2$ has free basis $u,v$ and we start by sending $t$ to $v$,
$x$ to $u^{jl}$ and $y$ to $u^{-il}$ so that the edge group 
$\langle x^iy^j\rangle$ is sent to the identity. Similarly we send
$a$ to $vu^{jl}v^{-1}$ and $b$ to $vu^{-jk}v^{-1}$. We can now send all other
vertex groups in the loop $\Lambda$
to the identity.

Having defined in either case our homomorphism on $L(\Lambda)$,
back in $\Gamma$ we may also have 
trees emanating from the vertices in the
loop, so we extend the definition by sending all vertex groups in a tree
to zero if its root vertex group was sent there. This now leaves the two
trees $T_x$ with root $\langle x,y\rangle$ and $T_a$ with
$\langle a,b\rangle$ (or just the tree $T_x$ in the first case). 
For each of these
we can define homomorphisms $h_x,h_a$
to $\Z$ with $h_x(x)=h_a(a)=jl$ and $h_x(y)=-il$, $h_a(b)=-jk$
using Lemma \ref{hom} (or $h_x(x)=0, h_x(y)=1$ in the first case).
We then compose
so that an element $g_x$ in the fundamental group formed from $T_x$
is sent to $u^{h_x(g_x)}\in F_2$ and likewise any
$g_a$ in the fundamental group formed
from $T_a$ goes to $vu^{h_a(g_a)}v^{-1}$. 
This now allows us to extend our homomorphism to the whole of $G(\Gamma)$,
even when $\Gamma$ is infinite,
and we have that the image is a non abelian subgroup of $F_2$, so certainly
$G$ surjects to $F_2$. 
\end{proof}

However this result above requires the vertex groups to equal $\Z^2$ and the
edge groups to equal $\Z$, not subgroups of these which will be
encountered when dealing with an arbitrary subgroup $S$ of $G(\Gamma)$.
We now deal with this case, though we will need to avoid some subgroups
of $G$ which, as well as the situation when $\Gamma$ contains self loops,
will require dropping to a finite index subgroup $H$ of $G$ before 
establishing that the strongest Tits alternative holds for $H$. 
\begin{prop} \label{nearm} 
Suppose that $G(\Gamma)$
is a finite graph of groups having $\Z^2$ vertex groups 
and $\Z$ edge groups with $\Gamma$ containing no self loops,
and such that all inclusions
of edge groups in vertex groups are maximal. 
Let $S$ be any subgroup of $G$. Then either $S$ surjects to $F_2$
or is free abelian (therefore isomorphic to $\{id\}$, $\Z$ or $\Z^2$)
or the graph of groups $S(\Sigma)$ obtained by setting 
$\Sigma=S\backslash T$ for $T$ the Bass - Serre tree of $G$ results in
$\Sigma$ being a graph with a single loop $\Lambda$ which is not a self
loop, and such that all edges in $\Lambda$ are of $\Z$ type, whereas at most
one vertex in $\Lambda$ is of $\Z^2$ type, with all other vertices in
$\Lambda$ of $\Z$ type.  
\end{prop}
\begin{proof}
Let $S$ be any subgroup of $G$,
so that as mentioned it is
also a graph of groups $S(\Sigma)$ with vertex stabilisers equal
to $S\cap\Z^2$ and edge stabilisers equal to $S\cap\Z$. As before
we are
done unless it turns out that $\Sigma$ contains a single loop $\Lambda$.
Moreover as $\Gamma$ does not contain
any self loops, we conclude that $\Sigma$ does not either. This is because if
we consider the action of $G$ on its Bass - Serre tree $T$ with quotient
graph $\Gamma=G\backslash T$, an element in $S$ which identifies the
endpoints of an edge in $T$ will also lie in $G$, and $G$ acts without
edge inversions. Now $S(\Sigma)$ may have vertex groups which are not
$\Z^2$ and edge groups which are not $\Z$, but we still have maximal
inclusions of edge groups into vertex groups as before.

First suppose that there is an $\{id\}$ type edge lying in $\Lambda$. Then
removal of this edge from $\Sigma$ gives us a graph of groups where the
graph is now a tree, and this group will surject to $\Z$ by Lemma \ref{hom}
(unless every vertex group is trivial, in which case we have the trivial
group). Now on putting back the missing edge, we find that $S$ is isomorphic
to the free product of this group with $\Z$ because here the edge group is 
trivial. So in this case either $S$ surjects to $F_2$ or it is equal to
$\{id\}*\Z$. 

Now we are left with the case where the loop $\Lambda$
has no $\{id\}$ type edges (thus no $\{id\}$ type vertices),
though there could well be $\Z$ type vertices. However whenever this
occurs then at such a vertex all edge inclusions will equal the vertex group
by maximality. 
Moreover we are assuming that there are at least 2 vertices of $\Z^2$ type
in $\Lambda$. We now tidy up the loop $\Lambda$ in a similar fashion
to that in the proof of
Theorem \ref{onel}, so that by contracting edges
we remove all $\Z$ type vertices from $\Lambda$ to form $\Lambda'$,
which will not change the subgroup $S$. This leaves us with
two graphs of groups $L(\Lambda')$ and $S(\Sigma')$ 
for new graphs $\Lambda',\Sigma'$ but the same
underlying groups $L,S$. Now in $\Lambda'$
all vertices are of $\Z^2$ type, all edges of $\Z$ type and there is
a single loop that is not a self loop (because the two
vertices of $\Z^2$ type in $\Lambda$ remain as distinct
vertices in the new loop). This certainly  
allows us to apply Proposition \ref{sloop} to $L(\Lambda')$, but also to
$S(\Sigma')$ because it is easily seen that
the proof of Proposition \ref{sloop} does not require
any restriction on the types of vertices or edges that lie outside the
loop $\Lambda$.
\end{proof} 
  
Thus we are left with $S(\Sigma)$ having a single loop $\Lambda$
with at most one $\Z^2$ type vertex in $\Lambda$ and all
others of $\Z$ type. Although this is now our only remaining case, the
problem is that in general these will not satisfy the conclusions of
our proposition, even if $\Lambda$ is not a self loop. 
For instance let us suppose that the graph of groups
$S(\Sigma)$ is such that $\Sigma$ consists only of
the loop $\Lambda$, with the vertices labelled $v_0,v_1,\ldots ,v_{n-1},
v_n=v_0$ in order, and such that all of $v_1,\ldots ,v_{n-1}$
are type $\Z$ vertices. Thus on forming the underlying group $S$,
we can contract $\Lambda$ to a self loop at $v_0$, so if this is
a $\Z^2$ type vertex with vertex group $\langle x,y\rangle$ and neighbouring
edge groups $\langle x^ky^l\rangle$ towards $v_{n-1}$ and $\langle x\rangle$
towards $v_1$
say, we see that $S$ is of the form 
$\langle x,y,t|[x,y],tx^{\pm 1}t^{-1}=x^ky^l\rangle$
which will not surject to $F_2$ unless $l=0$. On the other
hand if $v_0$ is of type $\Z$
with vertex group $\langle x\rangle$ then we would have 
$S=\langle x,t|txt^{-1}=x^{\pm 1}\rangle$ which can be the Klein bottle
group. Although these examples could certainly occur as subgroups of
our original group $G$, the remainder of this section involves showing
that they do not appear in the finite index subgroup of $G$
that we will now take.

\begin{prop} \label{doub}
Suppose that $G(\Gamma)$ is a graph of groups 
with $\Z^2$ vertex groups and $\Z$ edge groups, with $G$ acting on the
Bass - Serre tree $T$ to form the finite graph $\Gamma=G\backslash T$ with
natural projection $\pi_G:T\rightarrow\Gamma$. 
Then if $b$ is the first
Betti number of $\Gamma$, there exists a subgroup $H$ of index $2^b$
in $G$, thus which is also a graph of groups $H(\Delta)$ with
 $\Z^2$ vertex groups and $\Z$ edge groups
for $\Delta=H\backslash T$ and natural projection $\pi_H:T\rightarrow\Delta$,
that gives rise to a finite cover $q:\Delta\rightarrow\Gamma$ of these graphs
where $\pi_G=q\pi_H$ and with the following property:\\
Suppose that we have a closed path $p$ in $\Delta$
and consider its image $q(p)$ in $\Gamma$ which is also closed.
Then in tracing out the path $q(p)$ we pass
through every (non oriented)
edge an even number of times. In particular there are
no self loops in $\Delta$.
\end{prop}
\begin{proof}
We have the natural homomorphism $h$ from $G$ to the fundamental group
$\pi_1(\Gamma)$ of the underlying graph, which is free of rank $b$,
formed by quotienting out by all vertex and edge stabilisers. Thus
there is also a natural homomorphism factoring through this from $G$
to $(C_2)^b$ and the kernel $H$ of this is an index $2^b$ subgroup of $G$
containing all vertex and edge stabilisers of $\Gamma$. Consequently
$H$ can also be considered as
the fundamental group of a graph of groups $H(\Delta)$ for 
$\Delta=H\backslash T$. Now as $H$ is normal in $G$ we have that the
action of $G$ on $T$ with quotient $\Gamma$ factors through $\Delta$
to give us a map $q:\Delta\rightarrow\Gamma$ such that $\pi_G=q\pi_H$.
Moreover the action of the group $G/H$ on $\Delta$ with quotient $\Gamma$
is without fixed points because $H$ contains all stabilisers, thus $q$
is a regular covering map. We then have, say by \cite{htch} 
Proposition 1.40 (c), that
$G/H$ is isomorphic to $\pi_1(\Gamma)/q(\pi_1(\Delta))$, therefore the
subgroups $h(H)$  and $q(\pi_1(\Delta))$ are both normal in $\pi_1(\Gamma)$
with quotient $(C_2)^b$ so must be equal.

Given a finite graph $\Gamma$,
the cycle space of $\Gamma$ is the vector subspace of the edge space
(functions from the edges of $\Gamma$ to $C_2$) which is the span of the
closed reduced paths in $\Gamma$. Now an element $\gamma\in\pi_1(\Gamma)$
can be thought of as first a concatenation of closed paths corresponding to
the generators of $\pi_1(\Gamma)$ which is then reduced. But this reduction
does not change the parity of the number of times an edge is passed through,
thus we obtain a function from $\pi_1(\Gamma)$ to the edge space with
image equal to the cycle space, which is isomorphic to $(C_2)^b$. Moreover
this is a homomorphism from $\pi_1(\Gamma)$ onto $(C_2)^b$
so its kernel must equal the characteristic subgroup $h(H)$. Thus 
although our path $p$ might not be reduced, it can still be thought of as
an element of $\pi_1(\Delta)$ where $q(p)$ lies in 
this kernel $h(H)$, without changing the parity of the number of edge visits.
Hence we have that
the edge set of $q(p)$ induces the zero map from the edges of $\Gamma$
to $C_2$, so every edge in $\Gamma$ has been passed through an
even number of times.
\end{proof}

We now have our finite index subgroup of $G$ in place, allowing us to
present our main result of this section. The proof is all about ensuring
there is a finite index subgroup which
avoids the examples given after Proposition \ref{nearm}.
\begin{thm} \label{main} 
Suppose that $G(\Gamma)$
is a finite graph of groups having $\Z^2$ vertex groups 
and $\Z$ edge groups and such that all inclusions
of edge groups in vertex groups are maximal. 
Then there is a finite index subgroup $H$ of $G$ such that for every
subgroup $S$ of $H$ we have that either $S$ surjects to $F_2$ or $S$
is free abelian (therefore isomorphic to $\{e\}$, $\Z$ or $\Z^2$).
\end{thm}
\begin{proof}
We will assume that we have obtained
$H(\Delta)$ from $G(\Gamma)$ as in the statement of Proposition \ref{doub}.
Now suppose that $S$ is any subgroup of $H$, which
itself has finite index in $G$. We have that the action of $G$ on its
Bass - Serre tree $T$ with projection $\pi_G:T\rightarrow \Gamma=G\backslash T$
gives rise to the graph of groups $G(\Gamma)$, but
by restricting this action on $T$ to the subgroup $S$ we obtain the graph
of groups $S(\Sigma)$ for $\Sigma=S\backslash T$. Now by earlier results
in this section, we can assume that the quotient graph
$\Sigma$ (which could be infinite)
has a single loop $\Lambda$. Moreover this is not a self loop 
because there are none in $\Delta$, so by Proposition \ref{nearm} we are
done unless the vertices of $\Lambda$
are all of $\Z$ type apart from possibly one of $\Z^2$ type, and such that
the edges of $\Lambda$ are all of $\Z$ type, with maximal inclusions into
the vertex subgroups of $\Lambda$. Our method of proof is to
lift $\Lambda$ to a path $p$ in $T$ and then give a
parity argument utilising the fact that $S$ is a subgroup of $H$ in order to
get a contradiction.

We will proceed by labelling  
the vertices in $\Lambda$ consecutively as $v_0,v_1,\ldots ,v_n$ $=v_0$
with $v_0, v_1$ joined by
the edge $e_0$, and so on for edges $e_1$ up to $e_{n-1}$ which joins 
$v_{n-1}$ to $v_n=v_0$. We assume that if there is one vertex of $\Z^2$
type then it is $v_0$, with all others of $\Z$ type. 
Now as $\Sigma=S\backslash T$, we can take a vertex
$t_0\in T$ which lies above $v_0$ and which will have stabiliser in $S$
equal to
the group label on the vertex $v_0$ in the graph of groups $S(\Sigma)$
which we have obtained from the action of $S$ on $T$.
There will next be an edge $d_0$ in $T$ joining $t_0$ that
lies above $e_0$ and also with stabiliser in
$S$ equal to the label on its image in $\Sigma$, given as a subgroup of
the stabiliser in $S$ of the vertex $v_0$. 
We continue this process until 
we have found vertices $t_0,\ldots ,t_{n-1},t_n$ above $v_0,\ldots ,v_{n-1},
v_n=v_0$ and 
edges $d_0,\ldots ,d_{n-1}$
in $T$ above $e_0,\ldots ,e_{n-1}$ so that $d_i$ joins the vertices
$t_i$ and $t_{i+1}$. Note that although 
$t_n$ also lies above $v_0$ it is
not equal to $t_0$ because $T$ is a tree, 
though there will be an element $s\in S$ 
with $s(t_0)=t_n$ and so the stabiliser (in $S$ or in $G$)
of $t_n\in T$ will be the
conjugate by $s$ of the stabiliser (in $S$ or in $G$) of $t_0$.
We put an orientation on the straight line path from $t_0$ to $t_n$,
which we will call $p$,
so that the edge $d_i$ points from $t_i$ to $t_{i+1}$ and this induces
an orientation on the loop $\Lambda$. We also take another lift of the
edge $e_{n-1}$ which we call $d_{-1}$, this time 
joining $s^{-1}(t_{n-1})$ and $t_0$ and such that its stabiliser in $S$ is the
subgroup of the vertex stabiliser in $S$ of $t_0$ which is the edge label in
$S(\Sigma)$ on $e_{n-1}$ lying nearest to $v_0$. 

We next consider the image under $\pi_G$ of our path $p$.
As the element $s\in S$ sending $t_0$ to $t_n$
is also in $G$, we know
$\pi_G(p)$ is clearly a closed
path in $\Gamma$ but it need not be reduced. For instance
there can be backtracks in $\pi_G(p)$ when say $t_{i-1}$ and $t_{i+1}$ are
identified in $\Gamma$ by an element of the vertex stabiliser 
of $t_i$ in $G$ which does not lie in $S$. But as our subgroup $S$
lies in the finite index subgroup $H$ of $G$, we have that 
this path is the image
under the covering map $q$ of the closed path $\pi_H(p)$
in $\Delta$. Thus by Proposition \ref{doub} every (non oriented)
edge is passed through an even number of times when tracing out $\pi_G(p)$
in $\Gamma$.

Our purpose now is to note the relationship between the 
vertex and edge stabilisers of $G$ acting on $T$ and the vertex and edge
groups in the graph of groups $G(\Gamma)$ which we are thinking of as
labels appearing on the vertices and edges of $\Gamma$.
We only need this information for the path $p$, so to this
end  we define $V_i$ (where $0\leq i\leq n-1$) to be the group
$V_{\pi_G(t_i)}$ at the vertex $\pi_G(t_i)\in\Gamma$. Similarly given the edge
$d_j$ for $0\leq j\leq n-1$ with orientation as above, we set $E_j$ to be
the edge subgroup $E_{\pi_G(d_j)}$ of $V_{\pi_G(t_j)}$ which is the outgoing
edge label of $\pi_G(d_j)$ in $G(\Gamma)$ and $F_{j+1}$ 
equal to the edge subgroup $F_{\pi_G(d_j)}$ of $V_{\pi_G(t_{j+1})}$ by reading
the incoming edge label. In particular both $E_i$ and $F_i$ are subgroups of
$V_i$ for $0<i<n$ and $E_0,F_n$ are subgroups of $V_0$.

In order to form the fundamental group $G$ from the graph of groups
$G(\Gamma)$, we first
pick a maximal tree in $\Gamma$ and then associate the remaining
edges (once given an arbitrary orientation) to a free basis of the 
fundamental group $\pi_1(\Gamma)$.
Earlier we took a lift such that the initial vertex $t_0\in T$
actually has stabiliser equal to the label $V_0$ at the vertex
$\pi_G(t_0)\in\Gamma$.
Then we have that the vertex stabiliser of the point $t_i$, 
the edge stabiliser of the edge $d_{i-1}$ coming into $t_i$, and
the edge stabiliser of the edge $d_i$ leaving $t_i$ are all
simultaneously conjugate in $G$ to the labels $V_i,F_i,E_i$ respectively
of their images under $\pi_G$ in the graph of groups $G(\Gamma)$.
In particular we have that $E_i$ and $F_i$ are equal subgroups of $V_i$
if and only if the stabilisers $\mbox{Stab}_G(d_{i-1})$ and
$\mbox{Stab}_G(d_i)$ are equal subgroups of $\mbox{Stab}_G(t_i)$.

We now look more closely at the stabilisers of $G$ acting on $T$.
Thus let us set  
the vertex stabiliser in $G$ of the point $t_i\in T$ (where $0\leq i\leq n)$ 
to be
$\langle x_i,y_i\rangle$ and let the (infinite cyclic) edge stabiliser
in $G$ of the edge $d_j$ for $0\leq j\leq n-1$
be as follows: by changing bases if necessary
we can assume by maximality that it is equal to the
cyclic subgroup $\langle x_j\rangle$ of $\mbox{Stab}_G(t_j)$,
although this edge stabiliser is also a maximal cyclic subgroup of
the vertex stabiliser 
$\langle x_{j+1},y_{j+1}\rangle$ of $t_{j+1}$. 

Next we perform a similar comparison between the
vertex and edge labels on $S(\Sigma)$ and the stabilisers of $S$ acting
on $T$, at least for the loop $\Lambda$ in 
$\Sigma$. Of course the latter are found by simply intersecting
$S$ with the relevant vertex or edge stabiliser of $G$ acting on $T$. 
Thus the stabiliser in $S$ of $t_i$ for $0\leq i\leq n-1$
will equal $S\cap\langle x_i,y_i\rangle$, and we know these groups
are all infinite cyclic, except possibly for $v_0$ which could have $\Z^2$ 
type. Also the edge stabilisers $S\cap\langle x_j\rangle$ in $S$ of $d_j$ for
$0\leq j\leq n-1$ are all of $\Z$ type, thus
we will have 
$S\cap\langle x_j\rangle=\langle x_j^{a_j}
\rangle$ for some $a_j\neq 0$. 
However the stabiliser in $S$ of the edge $d_{j-1}$, which
is also of $\Z$ type,
must therefore be an infinite cyclic subgroup of the vertex 
stabiliser $S\cap\langle x_j,y_j\rangle$ and this latter
subgroup is equal to $\langle x_j^{a_j}\rangle$ when $1\leq j\leq n-1$
because these vertices are also of $\Z$ type.
This implies by maximality in $G(\Gamma)$ that $\langle x_j\rangle$ is 
equal to the incoming stabiliser in $G$ of $d_{j-1}$.
But as the incoming and outgoing stabiliser of the edge $d_j$ in $T$ are
the same group, we get
\[\langle x_0\rangle=
\langle x_1\rangle=\langle x_2\rangle=\ldots =\langle x_{n-1}\rangle,\]
thus by replacing elements with their inverse if necessary we have
$x_0=x_1=\ldots =x_{n-1}=x$ say
and $a_0=a_1=\ldots =a_{n-1}=a$ say. In particular
the element $x$
stabilises the geodesic line between $t_0$ and $t_n$ as it lies in all of the
relevant edge and vertex stabilisers.
However if the vertex $v_0$ is of $\Z^2$ type then
we need not have the incoming edge stabiliser of $d_{-1}$ in $S$
equal to the
outcoming edge stabiliser $S\cap \langle x\rangle=\langle x^a\rangle$ 
of $d_0$ at
the vertex $t_0\in T$ but instead 
the incoming edge stabiliser of $d_{-1}$
would equal $\langle x_0^ky_0^l\rangle$ in $G$ say,
and $\langle x_0^{mk}y_0^{ml}\rangle$ say in $S$.

At last we are ready to compare our path $p$ from $t_0$ to $t_n$
in $T$ with its image $\pi_G(p)$ in $\Gamma$. We noted earlier that
as we walk along the path $p$ or $\pi_G(p)$, 
a vertex having the property that the inclusions of the two edge groups
are equal in this vertex group
will display this behaviour both in $T$ and in $G(\Gamma)$.

Let us first consider the case where the vertex $v_0$ has $\Z^2$ type.
Here we know that the final incoming edge group $F_n$ in $G(\Gamma)$
and the initial outgoing edge group $E_0$ are both subgroups of the
vertex group $V_0$ but that they need not be equal.
However we have that the subgroups $F_i$ and $E_i$ are equal in $V_i$ for all
other vertices because of the correspondence with the stabilisers.
If in fact we do have $E_0=F_n$ then as the vertex $v_0$ has
$\Z^2$ type, the first case of the proof of Proposition \ref{sloop}
immediately applies here for $S(\Sigma)$, telling us that $S$ surjects $F_2$.
Otherwise we now define an equivalence relation on
the set $X$ of edges $d_i$ ($0\leq i\leq n-1$)
in our path $p$ for which the image $\pi_G(d_i)$ (here regarded as a
non oriented edge) has an endpoint equal to
$\pi_G(t_0)$ as follows: first suppose that there are no self loops
in $\pi_G(p)$. Then
for each edge $d_i$ in $X$
we let the subgroup $G_i$ of $V_0$ 
be whichever one of the edge subgroups $E_i$ or $F_{i+1}$ of $\pi_G(d_i)$
lies nearest to $\pi_G(t_0)$ (so if the pair of
edges $\pi_G(d_{i-1})$ and $\pi_G(d_i)$ both pass through $\pi_G(t_0)$ then
we will have $G_{i-1}=F_i$ and $G_i=E_i$ which are both subgroups of
$V_i=V_0$, whereas if it is the pair $\pi_G(d_i)$ and $\pi_G(d_{i+1})$
then now $G_i=F_{i+1}$ and $G_{i+1}=E_{i+1}$ which both lie in
$V_{i+1}=V_0$). Then if $d_i$ and $d_j$ are two edges in $X$, we say that
they are equivalent exactly when $G_i$ and $G_j$ are equal subgroups of
$V_0$. Now we know that $G_0=E_0$ and $G_{n-1}=F_n$ are not equal, so $d_0$ 
and $d_{n-1}$ must lie in
different equivalence classes. But every edge in $\Gamma$ 
is passed through an even number of times by Proposition \ref{doub}
and as $\pi_G(d_i)=
\pi_G(d_j)$ clearly implies that $G_i=G_j$ (even if the images of these
edges inherit opposite orientations from $p$) because we are reading the
same edge label from $G(\Gamma)$ in both cases, 
we must have that each equivalence class has even size.

However we will now look at the equivalence class $C$ of $d_0$ in a different
way by gradually building it up. Let us start by setting $C$ equal to
$\{d_0\}$ so that we have a set of odd size. Now the image of $d_0$ will
appear again as every edge in $\pi_G(p)$ is visited an even
number of times. Suppose then that we find $0<i\leq n-1$ with 
$\pi_G(d_i)=\pi_G(d_0)$. If $\pi_G(d_i)$ is travelling away from the vertex
 $\pi_G(t_0)$ then certainly $i\neq n-1$ because $\pi_G(d_{n-1})$ ends at
$\pi_G(t_0)$, and similarly $i\neq 1$ by considering $\pi_G(d_{i-1})$. In
this case we have $d_{i-1},d_i\in X$ with $G_{i-1}=F_i$ and $G_i=E_i$ and
we know that $E_i=F_i$ here. However we also have $G_i=G_0$ because 
$\pi_G(d_i)=\pi_G(d_0)$.

If however $\pi_G(d_i)$ travels towards $\pi_G(t_0)$ then we now have
$G_i=F_{i+1}$ but again $G_i=G_0=E_0$, therefore we cannot have $i=n-1$
because $G_{n-1}=F_n\neq E_0$. Thus $0<i<n-1$ and we also have $d_{i+1}\in X$
with $G_{i+1}=E_{i+1}=F_{i+1}=G_i=G_0$. Thus in either case we see that
two more edges $d_{i-1},d_i$ or $d_i,d_{i+1}$ are equivalent to $d_0$.
Consequently we now add them to $C$
but $|C|$ is still odd. Thus
some edge in $C$ must have its projection visited again
so this argument can be repeated, to find a new edge with its projection
having one end at $\pi_G(t_0)$. Moreover the edge immediately before or
after this edge that also projects to one with an endpoint at $\pi_G(t_0)$
will not yet have appeared in $C$ because at each stage we add to $C$
a pair of edges with image passing through  $\pi_G(t_0)$. Thus two more  
elements can now be added to $C$ and the argument can again be  
repeated until $C$ is of arbitrarily large but odd
size, which is a contradiction.

If there are self loops then the previous argument can now be made
to work as follows. If we obtain a self loop from $X$, namely the
edge $d_i$ is such that the initial and terminal endpoint of $\pi_G(d_i)$
are both equal to $\pi_G(t_0)$, then we split the edge $d_i$ into two
new edges $d_i,d_{i+1}$ (with the later edges renumbered accordingly).
We think of these as an outcoming edge $d_i$ where we set $G_i$ (in the new
numbering) equal to $E_i$ (in the old numbering) and $G_{i+1}$ (in the
new numbering) equal to $F_{i+1}$ (in the old numbering). Our parity
argument will now work as before.  

The remaining case is where all vertices and edges have $\Z$ type, so that
the fundamental group of the graph of groups obtained from the loop $\Lambda$ 
in $\Sigma$ will be isomorphic either to $\Z^2$ or to the Klein bottle
group and we will now eliminate the latter under the assumption that this
group lies in $H$. Consequently we have $E_i=F_i$ for $1\leq i\leq n-1$ 
as before, and now also $F_n=E_0$. The image in $\Gamma$ under $\pi_G$
of the path $p$  can be thought of as a connected subgraph $\Pi$ of
$\Gamma$. We have that all edge groups in $\Pi$ are infinite cyclic and
for the edge $d_i$ we know that the outgoing edge stabiliser label
$E_i$ in $G(\Gamma)$ is a subgroup of $V_i$ and the incoming label
$F_{i+1}$ is a subgroup of $V_{i+1}$ (with $F_n\leq V_0$). Thus for each
vertex $w\in\Pi$ and then for each edge $e$ in $\Pi$ that leaves
$w$, we have that the outward pointing edge stabiliser $E_e$ is an
infinite cyclic subgroup and we choose a preferred generator $g_e$ of
  $E_e$. However in our graph of groups $G(\Gamma)$ we do not just have
group labels $E_i$ and $F_{i+1}$ at each end of the edge $\pi_G(d_i)$, which
we have used up to now, but also a given isomorphism of the edge stabiliser
to each of the infinite cyclic subgroups $E_i$ and $F_{i+1}$, thus
providing a given isomorphism between them and hence between $E_i$
and $E_{i+1}=F_{i+1}$. We now label the edge $d_i$ with $+$ or $-$ as follows:
by looking at $\pi_G(d_i)$ in $\Pi$, we will find that the given isomorphism
from $E_i$ to $E_{i+1}$ either sends the preferred generator of $E_i$ to
that of $E_{i+1}$, in which case the edge $d_i$ is labelled $+$, or to
its inverse in which case $d_i$ is labelled $-$. In particular if
$\pi_G(d_i)=\pi_G(d_j)$ then (regardless of orientation) $d_i$ and $d_j$
have the same label, because these are chosen after applying $\pi_G$.

Thus on forming the subgroup $S$ from $S(\Sigma)$, we can first consider
the loop $\Lambda$. We have that $S$ conjugates the stabiliser
$\langle x^a\rangle$ in $S$ of the vertex $t_0$ to that of $t_n$
which is also equal to $\langle x^a\rangle$, but in order to work out whether
we have $sx^as^{-1}=x^a$ or $sx^as^{-1}=x^{-a}$, we follow the image path
$\pi_G(p)$ in $\Pi$, whereupon we need to keep track of how the element
$x$, which lies in all of $\mbox{Stab}_G(t_0),\ldots ,\mbox{Stab}_G(t_{n-1}),
  \mbox{Stab}_G(t_n)$, relates to these generators $g_i$. Now by the
description earlier of the relationship between stabilisers in $G$
and labels in $G(\Gamma)$, we can take $x=g_0$ where we are setting
$g_i$ equal to the preferred generator of $E_i$. Then as we walk along $p$
we have at each vertex $t_i$ that $x$ is conjugate in $G$ to one of
$g_i^{\pm 1}$, and we can keep track of the sign when
moving to $t_{i+1}$ by changing it 
if $\pi_G(d_i)$ is labelled with $-$, so that $g_i^{\pm 1}$ is
replaced by $g_{i+1}^{\mp 1}$ respectively,
but keeping the sign if $\pi_G(d_i)$ has a $+$ so that $g_i^{\pm 1}$
becomes $g_{i+1}^{\pm 1}$.
When we arrive at $t_n$ our
element will be $g_n$ or $g_n^{-1}$ according to the parity of the number
of negative edges we walked over in $\Pi$ and this element
$g_n^{\pm 1}$ is a conjugate (by $s$) of $x=g_0$. But $g_n=g_0$
because they are both the same preferred generator as $\pi_G(t_n)=\pi_G(t_0)$.
Thus we have $x^{\pm 1}=sxs^{-1}$ but as every edge is walked over
an even number of times, we conclude that $x^{\pm 1}$ is actually equal to $x$. 

Consequently in the graph of groups $S(\Sigma)$ we find that the loop
$\Lambda$ gives rise to the subgroup 
$\langle x,s|sx^as^{-1}=x^a\rangle\cong\Z^2$
of $S$. As for $S$ itself,
we are now left with trees emanating from here, so we will find
that either $S$ surjects to $F_2$ or is just $\Z^2$ by the
proof of Theorem \ref{onel} as if $\Sigma$ were a tree.
\end{proof}

\section{Other properties of tubular groups}
A group is said to be large if it has a finite index subgroup that
surjects $F_2$. Thus any group which virtually satisfies the strongest
Tits alternative and which is not virtually abelian, such as 
virtually special groups or the tubular
groups in the last section with maximal edge inclusions, is automatically
large. However it is not hard to see that actually all tubular groups
are large, including those containing non Euclidean or even non residually
   finite Baumslag - Solitar groups. Here we give a proof of this fact using
results from the last section which shows that tubular groups come
very close to always surjecting $F_2$. 
\begin{prop} \label{lrg}
If $G(\Gamma)$ is a tubular group and $\Gamma$ is not just a single vertex
(whereupon $G\cong \Z^2$) or a single vertex with a single self loop then
$G$ surjects $F_2$.

If $\Gamma$ is a single vertex with a single self loop, so that without
loss of generality we have
\[G=\langle x,y,t|[x,y],tx^it^{-1}=x^jy^k\rangle\]
for $i\neq 0$ and $j,k$ not both 0 then $G$ surjects $F_2$ if and only
if $k=0$. If not then $G$ has a subgroup of index 2 surjecting $F_2$.
\end{prop}
\begin{proof}
We know by the previous section that we are done if $\Gamma$ has at least
two loops, or if $\Gamma$ is a tree by the first part of the proof of 
Theorem \ref{onel} which does not require maximal inclusions of edge groups
because $\Gamma$ is a finite tree. The same point holds using the proof
of Proposition \ref{sloop} if $\Gamma$ has a single loop which is not
a self loop, on replacing
any edge inclusion with the maximal cyclic subgroup of the relevant
vertex group which contains it.
Thus suppose that
\[G=\langle x,y,t|[x,y],tx^it^{-1}=x^jy^k\rangle\]
so that if $k=0$ then we can remove $x$ to surject $F_2$. But in any
homomorphism of $G$ onto $F_2$, we know that $x$ and $y$ have to map
into the same cyclic subgroup of $F_2$, so that this map would factor
through
\begin{eqnarray*}
&&\langle x,y,t,z|[x,y],tx^it^{-1}=x^jy^k,x=z^l,y=z^m\rangle
\mbox{ for some }l,m\mbox{ not both }0\\
&\cong&\langle t,z|tz^{il}t^{-1}=z^{jl+km}\rangle
\end{eqnarray*}
which is 2-generated but not isomorphic to $F_2$ unless $l=0$ and $k=0$.

However as $\Gamma$ is a self loop then, even though $G$ does not surject
$F_2$ in this case, we can make an immediate application of Proposition
\ref{doub} with $b=1$ to obtain the graph of groups $H(\Delta)$ where the
index 2 subgroup $H$ of $G$ surjects $F_2$ by the start of this proof,
because $\Delta$ is not just a single self loop.
\end{proof}

As for group theoretic properties which are held by all free by $\Z$
tubular groups, as opposed to all tubular groups, we also have residual
finiteness (which in fact holds for all free by $\Z$ groups). The other
property, also shared by virtually special groups, 
which we will show here is that of being virtually
biorderable. This will be a consequence of 
\cite{perrlf03} Corollary 2.2 and Theorem 2.6, which together imply that
if the free by $\Z$ group $G=F_n\rtimes_\alpha\Z$ is such that all eigenvalues
of the matrix given by the action of $\alpha$ on the abelianisation
$\Z^n$ 
are real and positive then $G$ is biorderable.
\begin{thm} \label{vbi}
If $G(\Gamma)$ is a tubular free by $\Z$ group then 
$G$ is virtually biorderable.
\end{thm}
\begin{proof} We assume that we have a homomorphism $\chi:G\rightarrow\Z$
as described in Theorem \ref{fbyc} which is non zero on all edge groups.
This gives rise to a particular decomposition of $G$ as 
$F_n\rtimes_\alpha\Z$ for $F_n$ the kernel of $\chi$ and the characteristic
polynomial of the abelianised matrix is also equal to the Alexander
polynomial $\Delta_{G,\chi}(t)$. This can be evaluated using the Fox
calculus as described in many places: see \cite{map} for a treatment
of groups with a deficiency 1 presentation, as we have here.

We first suppose that $\Gamma$ is a tree of $v$ vertices, giving rise
to the standard $2v$ generator, $2v-1$ relator presentation of $G$ which
we read off from the graph of groups. We next obtain, using the Fox
calculus, the $2v-1$ by $2v$ Alexander matrix $A$ with entries in the
ring of Laurent polynomials $\Z[t^{\pm 1}]$, where $t$ generates the
image of $\chi$. We then delete each column in turn and take determinants
to obtain the $2v$ minors, with the Alexander polynomial (defined up to
units) being the highest common factor of these minors. We will show
by induction on $v$ that if the first column is deleted then the
corresponding minor is (up to units) a product of cyclotomic polynomials,
thus so is $\Delta_{G,\chi}$.
For $v=2$ let us take a basis $\langle a_1,b_1\rangle$ for the first
vertex, where the inclusion into $\langle a_1,b_1\rangle$ of the single
edge group lies in $\langle a_1\rangle$, so that $\chi(a_1)=p\neq 0$ and we
can choose $b_1$ so that $\chi(b_1)=0$. Similarly for the second vertex,
giving us a basis $\langle a_2,b_2\rangle$ where $\chi(b_2)=0$ and 
$\chi(a_2)=q\neq 0$. Thus our three relators here will be $[a_1,b_1]$,
$a_1^ka_2^{-l}$ (where $k,l\neq 0$ and $kp=lq$), 
and $[a_2,b_2]$. On ordering the generators
$a_1,b_1,a_2,b_2$ and the relators as above, we obtain
the Alexander matrix
\[\sma {cccc}
0&t^p-1&0&0\\
\frac{t^{pk}-1}{t^p-1}&0&\frac{1-t^{ql}}{t^q-1}&0\\
0&0&0&t^q-1\fma\]
thus our result holds so far. Now suppose this is true for $2v$ vertices
and we add a leaf to $\Gamma$, thus introducing two new generators and
columns $a_{v+1},b_{v+1}$, where again we assume that the inclusion
into $\langle a_{v+1},b_{v+1}\rangle$ of the new edge subgroup lies
in $\langle a_{v+1}\rangle$ and 
$\chi(a_{v+1})=r\neq 0,\chi(b_{v+1})=0$. If this new vertex joins
the $d$th vertex then we will have the two new relators (and consequent rows) 
$a_d^ib_d^ja_{v+1}^{-s}$ and $[a_{v+1},b_{v+1}]$. But as our new generators
$a_{v+1},b_{v+1}$ did not appear in any of the previous relations, this means
we are adding two new columns to our Alexander matrix, with the first
that corresponds to $a_{v+1}$ having zeros in all but the penultimate row
where we have $(1-t^{rs})/(t^r-1)$, whereas the second new column 
corresponding to $b_{v+1}$ has all zeros except $t^r-1$ at the very bottom.
Thus, regardless of what else lies in our new rows, on crossing off the
first column and taking the determinant we obtain
$1-t^{rs}$ times the previous determinant so our induction is complete.

If $\Gamma$ is not a tree then we now add extra edges to a maximal tree
$T$ with graph of groups $H(T)$ for which the above applies. But by the
description in the proof of Theorem \ref{fbyc}, we can replace
$\chi:G\mapsto\Z$ by the related homomorphism $\chi_1:G\mapsto\Z$
where $\chi_1$ sends the stable letters to zero but agrees with $\chi$
on the vertex groups. Then on adding the first edge we see from the
presentation (\ref{eq1}) in Theorem \ref{fbyc} and the description of the
Alexander polynomial in terms of the characteristic polynomial of the
abelianised matrix that $\Delta_{E,\chi_1}$ is the product of the
characteristic polynomial $\Delta_{H,\chi}$ as found above and the
determinant of the matrix with 1s along the lower diagonal and in the
top right hand corner, which is $\pm (t^n-1)$.

We then repeat this for all remaining stable letters, resulting in $G$
being expressed as $F_N\rtimes_\alpha\Z$ where the natural homomorphism
to $\Z$ has Alexander polynomial whose zeros are all roots of unity. But
on taking the finite cyclic cover $H=F_N\rtimes_{\alpha^i}\Z$ of index $i$
in $G$, we have that here the abelianised matrix is the $i$th power of
that for $\alpha$, so for suitable $i$ the eigenvalues are all 1. Thus $H$ 
is biorderable by the results quoted just before this theorem.
\end{proof}

We finish this section with some properties that tubular groups do not
possess. They are obviously not word hyperbolic but also are never
relatively hyperbolic by \cite{me} Corollary 5.5. However Theorem 4.2
in that paper shows they are nearly always acylindrically hyperbolic: a
tubular group is not acylindrically hyperbolic exactly when for every
vertex we have that all edge inclusions into that vertex group lie
in a single cyclic subgroup. (Acylindrical hyperbolicity implies the
property of being SQ universal but this is weaker than largeness which
we already have for tubular groups.)

The final property we consider is that of being residually free. As a 
subgroup of a residually free group is also clearly residually free, we 
have that any group known to be residually free will also immediately
satisfy the strongest Tits alternative. Therefore we should confirm
that tubular groups cannot be residually free (or virtually residually
free) in order to ensure that the results in the last section do not
follow from this property. However there is one type of tubular group
that is obviously residually free, namely anything isomorphic to
$F_n\times\Z$ and there are various graphs of groups with this as the
fundamental group.
\begin{prop} If $G(\Gamma)$ is a tubular group then $G$ is residually
free if and only if $G$ is isomorphic to $F_n\times\Z$ for some
$n\in\N$. This occurs if and only if\\
(i) For every vertex we have that all edge inclusions
are equal to a single maximal cyclic subgroup of this $\Z^2$ vertex group
and\\
(ii) We can assign an orientation of each of these maximal cyclic 
subgroups, namely a choice of one of the two possible generators, such that 
on following this generator around the edge subgroups of any closed loop
in the graph $\Gamma$, we return to the original orientation. 
\end{prop}
\begin{proof}
Here we only give a sketch. It can be shown without much problem, by mimicking
the approach of \cite{wlt} Theorem 2.9 for surface by $\Z$ groups, that
a free by $\Z$ group is residually free if and only if it is
isomorphic to $F_n\times\Z$ for some $n\in\N\cup \{0\}$. 
Thus first suppose that
we have a tubular group $G(\Gamma)$ where $G$ is residually free and 
$\Gamma$ is a tree, so that $G$ is free by $\Z$ from Corollary \ref{tree}
and hence isomorphic to $F_n\times\Z$. Moreover the same holds for any
subtree of $\Gamma$, from which we build up $G$ using amalgamated
free products. By \cite{mks} Corollary 4.5 we have that the centre
$Z(G)$ of $G=A*_CB$ is equal to $C\cap Z(A)\cap Z(B)$, so that by induction
on the number of edges $G$ has a
non trivial centre if and only if for every
vertex we have that all edge inclusions into that vertex group lie
in a single cyclic subgroup. However in $G\cong F_n\times\Z$ the centre
is a maximal cyclic subgroup, thus requiring that at every vertex
all edge inclusions equal the same maximal cyclic subgroup, as otherwise
we would have an element of a vertex group not lying in the centre of $G$
but a power of it would.

Next we assume that $\Gamma$ is a general finite graph but without self
loops. Then we can take different maximal trees of $\Gamma$, where every
edge will lie in one of these trees, and use the conclusion above in
combination with these various trees which will still force each edge
group at a vertex to equal the same maximal cyclic subgroup. We then
go back to our original maximal tree and add the remaining edges one by
one, whereupon we require the orientation condition in following the
edge subgroups around a loop, so that the relevant
stable letter commutes with
a generator of this edge group rather than conjugating it to its
inverse. Finally we can add on any self loops at a single vertex,
whereupon we still require the orientation condition.
\end{proof}
By \cite{asos} Lemma 3.9 we have that being acylindrically hyperbolic is a
commensurability invariant. Therefore suppose $G(\Gamma)$ is a tubular group
such that at some vertex in the graph the edge groups do not all lie in a
single cyclic subgroup of this vertex group. Then as mentioned $G$ is not
acylindrically hyperbolic, so it cannot be virtually residually free because
if so it is virtually $F_N\times\Z$ by the above and this
is certainly not acylindrically hyperbolic. Therefore the tubular groups
which virtually satisfy the strongest Tits alternative by consequence
of being virtually residually free are a very small collection amongst all
tubular groups.
 
\section{Examples of tubular groups}
In this final section we look at examples that have already appeared in the
literature and discuss their properties. Interestingly all graphs here
will consist of a single vertex with one or two self loops.
\hfill\\
{\bf Example 5.1}: The Burns, Karass, Solitar group\\
This group
$\langle a,b,t|[a,b],tat^{-1}=b\rangle$ was introduced in \cite{bks} as the
first 3-manifold group which was not subgroup separable (but then neither 
are many RAAGs).

However it is coherent and of cohomological and geometric dimension 2,
it does not contain non Euclidean Baumslag - Solitar subgroups, it is free
by $\Z$, it is acylindrically but not relatively hyperbolic, every non
trivial subgroup has a surjection to $\Z$, it is  biorderable, has the
unique product property, and has no zero divisors in its group algebra.
Further it acts freely on a CAT(0) cube complex; indeed freely and
cocompactly by \cite{wscb} Corollary 5.10. Moreover by \cite{liu} or
\cite{wsp} it is in fact virtually special and so has a finite index
subgroup satisfying the strongest Tits alternative (of which 2 is the
smallest such index by Theorem \ref{main} and Proposition \ref{lrg}). 
Being virtually special also implies that it is linear, even over $\Z$.
\hfill\\ 
{\bf Example 5.2}: The Gersten group\\
$G=\langle a,b,s,t|[a,b],sas^{-1}=a^{-1}b^2,tat^{-1}=b\rangle$ 
with two self loops contains the previous example and so is not
subgroup separable. However, in common with the Burns, Karass, Solitar group
above, it 
is coherent and of cohomological and geometric dimension 2,
it does not contain non Euclidean Baumslag - Solitar subgroups, it is free
by $\Z$, it is acylindrically but not relatively hyperbolic, every non
trivial subgroup has a surjection to $\Z$, it is  biorderable, has the
unique product property, and has no zero divisors in its group algebra.
However it was introduced in \cite{ger} as an example of a group which
does not act properly and cocompactly on a CAT(0) space. Also it is not
virtually special: we thank Mark Hagen for explaining this. If there was
a finite index subgroup $H$ of $G$ which was a subgroup of a RAAG (without
loss of generality finitely generated because $H$ is) then this RAAG acts
properly and cocompactly on the Salvetti complex, so $H$ acts properly on
a finite dimensional cube complex. Then we can use quasiconvex walls
obtained from $H$ 
to induce an action of $G$ on a CAT(0) cube complex, of higher but still
finite dimension. This will also be a proper action and finite dimensionality
tells us that $G$ acts properly and semisimply on this cube complex. However
\cite{bh} points out that Gersten's result extends to proper and
semisimple actions on a CAT(0) space, thus giving us a contradiction.
The Gersten group does act freely and hence properly on a CAT(0)
cube complex by \cite{wscb} but this will necessarily be infinite
dimensional.

Now by Theorem \ref{main} $G$ has a finite index subgroup (of index 4)
satisfying the strongest Tits alternative, yet we see that this property
cannnot be established for $G$ by a cubulation type argument. We believe
that $G$ is the best behaved example known, in terms of its abstract
group theoretic properties, that does not have a correspondingly
well behaved geometric action. However we do not know about its linearity,
over $\C$ or over $\Z$.
\hfill\\
{\bf Example 5.3}: The Woodhouse group\\  
In \cite{wd1} a criterion was given for when a tubular group acts freely
on a finite dimensional CAT(0) cube complex. We now look at Example 5.1
of that paper, which again comes from two self loops:
\[\langle a,b,s,t|[a,b],sabs^{-1}=a^2,tabt^{-1}=b^2\rangle\]
and which was shown there, despite being free by $\Z$ and so having a
free action on a CAT(0) cube complex by Theorem \ref{free}, to have no
such action on one which is finite dimensional.

As for its group theoretic properties, we again see that it
is coherent and of cohomological and geometric dimension 2,
it does not contain non Euclidean Baumslag - Solitar subgroups, it is free
by $\Z$, it is acylindrically but not relatively hyperbolic, every non
trivial subgroup has a surjection to $\Z$, it is left orderable, has the
unique product property, and has no zero divisors in its group algebra.
It does not have maximal edge inclusions but has a
cover of index 2 that does by Proposition \ref{fincv}. It is also virtually
biorderable and has a subgroup of index 8 satisfying the strongest Tits
alternative by Theorem \ref{main}. In the follow up paper \cite{wd2} it
is shown that for tubular groups, being virtually special is equivalent  
to acting freely on a finite dimensional CAT(0) cube complex, so like
the Gersten group this group does not virtually embed in a RAAG. Again we
do not know about its linearity.

We now consider some rather stranger examples.\\
\hfill\\
{\bf Example 5.4}: The Wise simple curve examples\\
Example 1.7 in \cite{wscb} introduced the following tubular groups:
\[G_n=\langle a,b,s,t|[a,b],sa^nbs^{-1}=ab,tab^nt^{-1}=ab\rangle\]
for $n\geq 2$ where all edge inclusions are maximal. Even so it was
shown there that they do not have equitable sets and so $G_n$ does
not act freely on any CAT(0) cube complex, thus it is not virtually
special. As in common with all tubular groups it 
is coherent and of cohomological and geometric dimension 2, is not
relatively hyperbolic,
it is locally indicable and hence left orderable, has the
unique product property, and has no zero divisors in its group algebra.
Moreover it does not contain non Euclidean Baumslag - Solitar subgroups,
and it is acylindrically hyperbolic.
However it is not free by $\Z$. Nevertheless, because all edge inclusions
are maximal, we can conclude that every non
trivial subgroup has a surjection to $\Z$ and that (by Theorem \ref{main})
there is a subgroup of index 4 that satisfies the strongest Tits
alternative. We do not know if $G_n$ is linear or virtually biorderable.
\hfill\\
{\bf Example 5.5}: The Wise non simple curve examples\\
Example 1.6 of \cite{wscb} was
\[ G_n=\langle a,b,s,t|sa^ns^{-1}=ab,tb^nt^{-1}=ab\rangle\]
so that $n=2$ is actually the Woodhouse group in Example 5.3. It was
shown in \cite{wscb} that $G_n$ does not act freely on a CAT(0) cube complex
if $n\geq 3$ and so we know it is not virtually special. Again we
conclude it is coherent and of cohomological and geometric dimension 2,
it does not contain non Euclidean Baumslag - Solitar subgroups,
it is acylindrically but not relatively hyperbolic,
it is locally indicable and hence left orderable, has the
unique product property, and has no zero divisors in its group algebra.
But as $G_n$ is not free by $\Z$
and the edge inclusions are not all maximal, we do not know if it
virtually satisfies the strongest Tits alternative.\\
\hfill\\
{\bf Example 5.6}: Wise's non Hopfian tubular group\\
The group
\[a,b,s,t|[a,b],sas^{-1}=a^2b^2,tbt^{-1}=a^2b^2\rangle\]
already appears in \cite{wsja} where it was shown to be non Hopfian, so
it is certainly not virtually special. Of course it is
coherent and of cohomological and geometric dimension 2,
it does not contain non Euclidean Baumslag - Solitar subgroups,
it is acylindrically but not relatively hyperbolic,
it is locally indicable and hence left orderable, has the
unique product property, and has no zero divisors in its group algebra.
Moreover \cite{wscb} shows that it acts freely on a CAT(0) cube complex,
though this cannot be finite dimensional by \cite{wd2}. However the edge
inclusions are not maximal and it is not free by $\Z$. Indeed as it is
not residually finite, it is not even virtually free by $\Z$.

\Address

\end{document}